\documentclass{amsart}

\usepackage{amssymb, amsmath}
\usepackage{mathrsfs}
\usepackage{amscd}
\usepackage{verbatim}

\allowdisplaybreaks

\usepackage{hyperref}

\theoremstyle{plain}
\newtheorem{theorem}{Theorem}[section]
\theoremstyle{remark}
\newtheorem{remark}[theorem]{Remark}

\theoremstyle{plain}
\newtheorem{corollary}[theorem]{Corollary}
\newtheorem{lemma}[theorem]{Lemma}
\newtheorem{proposition}[theorem]{Proposition}
\newtheorem{definition}[theorem]{Definition}

\numberwithin{equation}{section}

\def\N{{\mathbb N}}
\def\Z{{\mathbb Z}}

\def\R{{\mathbb R}}
\def\K{{\mathbb K}}
\def\C{{\mathbb C}}
\def\T{{\mathbb T}}

\newcommand{\E}{{\mathbb E}}
\renewcommand{\P}{{\mathbb P}}
\newcommand{\F}{{\mathscr F}}
\newcommand{\A}{{\mathscr A}}

\newcommand{\g}{\gamma}

\renewcommand{\O}{\Omega}

\renewcommand{\Re}{\hbox{\rm Re}\,}

\newcommand{\D}{\mathscr{D}}

\newcommand{\lb}{\langle}
\newcommand{\rb}{\rangle}

\begin{document}

\title[Regularity of Gaussian white noise]{Regularity of Gaussian white noise on the $d$-dimensional torus}

\author{Mark Veraar}
\address{Delft Institute of Applied Mathematics\\
Delft University of Technology \\ P.O. Box 5031\\ 2600 GA Delft\\The
Netherlands} \email{M.C.Veraar@tudelft.nl}



\date\today

\thanks{The author was supported by a VENI subsidy 639.031.930
of the Netherlands Organisation for Scientific Research (NWO)}


\begin{abstract}
In this paper we prove that a Gaussian white noise on the $d$-dimensional torus has paths in the Besov spaces $B^{-d/2}_{p,\infty}(\T^d)$ with $p\in [1, \infty)$. This result is shown to be optimal in several ways. We also show that Gaussian white noise on the $d$-dimensional torus has paths in a the Fourier-Besov space $\hat{b}^{-d/p}_{p,\infty}(\T^d)$. This is shown to be optimal as well.
\end{abstract}

\maketitle

\section{Introduction}

In \cite{Bourgain} it has been proved that the Gibbs measures are invariant for the nonlinear Schr\"odinger equation. Building on these ideas, in \cite{QuVa} and also in \cite{Oh} it has been shown that the mean zero Gaussian white noise on the torus $\T$ is invariant for the periodic Korteweg-de Vries equation (KdV). To prove this one needs that (KdV) is well-posed for initial conditions from function spaces with a negative smoothness index, such as Sobolev spaces $H^{s,p}(\T)$ with $s<0$ and $p\in [1, \infty]$ and other classes of function spaces. Here a negative smoothness index $s$ is needed, because it is well-known Gaussian white noise is supported on $\bigcap_{s<-1/2} H^{s,p}\setminus H^{-1/2,p}$. It seems that the first results on the support of Gaussian white noise into this direction have been obtained in \cite{ReRo} for $p=2$ and in \cite{Kus} for other values of $p$. Note that both \cite{ReRo} and \cite{Kus} consider processes on $\R^d$ instead of the torus.

In many instances, Besov spaces are the right class of function spaces in order to prove sharp results. This is also the case for regularity results for paths of Brownian motion and other classical processes. Sharp results for such processes have been proved for instance in \cite{Cie2, CKR, Roy} using an equivalent wavelet definition of Besov spaces. In particular, in these papers it has been shown that Brownian motion $B:[0,1]\times\O\to \R$ satisfies
\[\P\big(B\in B^{1/2}_{p,\infty}(0,1)\big) = 1.\]
In \cite{HV, VerCor} this has been proved directly from the $L^p$-incremental definition of Besov spaces.
Formally, one could say that white noise in dimension one is given by $\dot{B}$, and therefore, Gaussian white noise is in $B^{-1/2}_{p,\infty}$ almost surely. In this paper we combine some of the ideas in \cite{CKR,HV} with Fourier analysis to obtain sharp results on the regularity of Gaussian white noise.


It might be helpful for non-experts to recall some elementary embedding results for Besov spaces and Sobolev spaces. Here we follow the standard notations as in \cite{TrSch,Tri}. Of course one has that $B^{s}_{2,2} = H^{s,2}$. This is no longer true for $B^{s}_{p,2}$ and $H^{s,p}$ with $p\neq 2$. One has the following embedding results (see \cite[2.3.3]{Tri} for $\R^d$ and \cite[Chapter 3]{TrSch} for $\T^d$)
\begin{equation}\label{eq:elememb}
\begin{aligned}
B^{s}_{p,2} & \hookrightarrow H^{s,p}\hookrightarrow B^s_{p,p} \ \ \text{ if $p>2$,  }
\\ B^s_{p,p} & \hookrightarrow  H^{s,p} \hookrightarrow B^{s}_{p,2} \ \  \text{ if $1\leq p<2$,}
\end{aligned}
\end{equation}
and, for any $\epsilon>0$ and $p,q,r,s\in [1, \infty]$,
\[B^{s+\epsilon}_{p,q}\hookrightarrow H^{s,r} \hookrightarrow B^{s-\epsilon}_{p,q}.\]

In the paper we consider the following question:
\begin{itemize}
\item On which Besov spaces is the $d$-dimensional Gaussian white noise $W:\O\to \T^d$ supported?
\end{itemize}
Our main result is that for all $p\in [1, \infty)$ one has
\[W\in B^{-d/2}_{p,\infty}(\T^d) \text{ almost surely.} \]
Moreover, we show that this is optimal in several ways. In particular, the known results on Sobolev spaces $H^{s,p}(\T^d)$ can easily be derived from our results.

Let us go back to the approach in \cite{Oh} to study the (KdV). In order to prove well-posedness of the (KdV) with a white noise initial condition, a new class of function space is introduced which is a Fourier-type Besov space denoted by $\hat{b}^s_{p,q}(\T)$ (see Section \ref{sec:FourierBesov}). An important step in the proof in \cite{Oh} of the invariance of Gaussian white noise for (Kdv) is that Gaussian white noise satisfies $W\in \hat{b}^s_{p,\infty}(\T)$ almost surely for all $s<-1/p$ and all $p\in [1, \infty)$. It is natural to ask what the optimal exponents $(s,p,q)$ are for which $W\in \hat{b}^s_{p,\infty}(\T)$ almost surely. Our main result here is that for all $p\in [1, \infty)$ one has
\[W\in \hat{b}^{-d/p}_{p,\infty}(\T) \text{ almost surely.}\]
Again this is optimal in several ways.

\medskip

\noindent {\em Acknowledgment} -- The author thanks Jan van Neerven for helpful comments.

After posting this paper on ArXiv,  \'Arp\'ad B\'enyi and Tadahiro Oh kindly pointed out that there is some overlap with their paper  \cite{BT}, in which similar techniques are used to characterize the exponents $(s,p,q)$ for which a Brownian motion on $\T$ is in $B^{s}_{p,q}(\T)$ and in $\hat{b}^{s}_{p,q}(\T)$. In fact, in dimension $d=1$ some (but not all) of our results could alternatively be deduced from theirs.

\section{Preliminaries}

We will write $a\lesssim b$ if there exists a universal constant
$C>0$ such that $a\leq Cb$, and $a\eqsim b$ if $a\lesssim b\lesssim
a$. If the constant $C$ is allowed to depend on some parameter $t$,
we write $a\lesssim_t b$ and $a\eqsim_t b$ instead.

\subsection{Besov spaces of periodic functions\label{sec:Besov}}
Let $\phi\in C^\infty(\R^d)$ be a fixed nonnegative function
$\phi$ with support in $\{t\in\R^d: \ \tfrac12\le
|t|\le 2\}$ and which satisfies
$$ \sum_{j\in\Z} \phi(2^{-j}t) =1 \quad\hbox{for $t\in \R^d\setminus\{0\}$}.$$
Additionally assume that $\phi(x) =1$ for all $2^{-1/2}\leq |x|\leq 2^{1/2}$.

Define the sequence $({\varphi_j})_{j\ge 0}$ in $C^\infty(\R^d)$ by
\[{\varphi_j}(t) = \phi(2^{-j}t) \quad \text{for}\ \  j=1,2,\dots \quad
\text{and} \ \ {\varphi_0}(t) = 1- \sum_{k\ge 1} {\varphi_j}(t), \quad
\xi\in\R^d.\]
Then all the functions $\varphi_j$ have compact supports.

Let $\T^d = [-\pi,\pi]^d$. Let $\F:L^2(\T^d)\to \ell^2(\Z^d)$ denote the Fourier transform, i.e.
\[(\F f)(k)  = \hat{f}(k) = (2\pi)^{-d}\int_{\T^d} f(x) e^{-ik\cdot x}\, dx,  \ \ f\in L^2(\T^d)\]

The space $\mathscr{D}(\T^d)$ is the space of complex-valued infinitely
differentiable functions on $\T^d$. On $\mathscr{D}(\T^d)$ one can define the
seminorms
\[\|f\|_{\alpha} = \sup_{x\in \T^d} |D^\alpha f(x)|,\]
where $\alpha = (\alpha_1, \ldots, \alpha_d)$ is a multiindex, and in this way
$\mathscr{D}(\T^d)$ is a locally convex space. Its dual space
$\mathscr{D}'(\T^d)$ is called the space of distributions. In particular, one has $g\in \mathscr{D}'(\T^d)$ if and only if there is a $N\in \N$ and a $c>0$ such that
\[|\lb f, g\rb| \leq  c \sum_{|\alpha|\leq N}|f\|_{\alpha}.\]
For details we refer to  \cite[Section 3.2]{TrSch}. In particular, recall the following two facts which we will not need, but are useful to support the intuition.
\begin{itemize}
\item Any function $f\in\mathscr{D}(\T^d)$ can be represented as
\[f(x) = \sum_{k\in \Z^d} a_k e^{ik\cdot x} \text{  in  } \mathscr{D}(\T^d),\]
with $(a_k)_{k\in \Z}$ scalars such that
\begin{equation}\label{eq:condakD}
\sup_{k\in \Z^d}(1+|k|)^m |a_k|<\infty \ \text{ for any $m\in \N$.}
\end{equation}
In this case, one has $a_k = \hat{f}(k)$ for each $k\in \Z^d$. Conversely, if $(a_k)_{k\in \Z^d}$ satisfies \eqref{eq:condakD}, then $\sum_{k\in \Z^d} a_k e^{ik\cdot x}$ converges in $\mathscr{D}(\T^d)$.
\item Any distribution $g\in\mathscr{D}'(\T^d)$ can be represented as
\[g = \sum_{k\in \Z^d} a_k e^{ik\cdot x} \text{  in  } \mathscr{D}'(\T^d),\]
with $(a_k)_{k\in \Z}$ scalars such that
\begin{equation}\label{eq:condakDprime}
\sup_{k\in \Z^d}(1+|k|)^{-m} |a_k|<\infty \ \text{ for some $m\in \N$.}
\end{equation}
In this case, one has $a_k = \hat{g}(k)$ for each $k\in \Z^d$. Conversely, if  if $(a_k)_{k\in \Z^d}$ satisfies \eqref{eq:condakDprime}, then $\sum_{k\in \Z^d} a_k e^{ik\cdot x}$ converges in $\mathscr{D}'(\T^d)$.
\end{itemize}

For a distribution $f\in \mathscr{D}'(\T^d)$
let $f_j\in \mathscr{D}(\T^d)$ be given by
\[f_j(x) = \sum_{k\in \Z^d} \varphi_j(k) \hat{f}(k) e^{ik\cdot x},  \ \ j\in \N.\]
By the properties of $(\varphi_j)_{j\geq 0}$ the series in $j$ has only finitely
many nonzero terms.
Observe that
\begin{equation}\label{eq:fjalsconv}
f_j(x) = \check{\varphi}_j * f(x) = \lb \check{\varphi}_{j,x}, f\rb
\end{equation}
where $\check{\varphi}_j(x) = \sum_{k\in \Z^d} \varphi_j(k) e^{ik\cdot x}$ and $\check{\varphi}_{j,x}(y) = \check{\varphi}_{j}(x-y)$.

\begin{definition}[Periodic Besov spaces]
Let $p,q\in [1, \infty]$.
Let
\[\|f\|_{B^s_{p,q}(\T^d)} := \Big(\sum_{j\geq 0} 2^{sjq} \| f_j |_{L^p(\T^d)}^q \Big)^{1/q},\]
if $q<\infty$, and
\[\|f\|_{B^s_{p,\infty}(\T^d)} := \sup_{j\geq 0}2^{sj}  \|f_j \|_{L^p(\T^d)}\]
if $q=\infty$.
The {\em Besov space $B^s_{p,q}(\T^d)$} is the space of all distributions $f\in \mathscr{D}'(\T^d)$ such that
$\|f\|_{B^s_{p,q}(\T^d)}<\infty$.
\end{definition}
One can show that the definition of $B^s_{p,\infty}(\T^d)$ does not depend on the
choice of $\phi$. Moreover, with two different functions $\phi$, the
corresponding norms in $B^s_{p,\infty}(\T^d)$ are equivalent (see
\cite[Section 3.5.1]{TrSch}).

\subsection{Vector-valued Gaussian random variables}

Let $(\O, \A, \P)$ be a probability space. Recall that $\gamma:\O\to \R$ is a {\em complex standard Gaussian random variable} if \[\gamma = \gamma_{\rm Re}/\sqrt{2}+i \gamma_{\rm Im}/\sqrt{2},\] where $\gamma_{\rm Re}$ and $\gamma_{\rm Im}$ are independent real standard Gaussian random variables (see \cite[Chapter 5]{Kal} for the definition of real Gaussian random variables). Let  $(\gamma_n)_{n=1}^N$ be a sequence of independent complex standard Gaussian. It is easy to check that the distribution $(\gamma_1, \ldots, \gamma_N)$  is invariant under unitary transformations and $\sum_{n=1}^N a_n \gamma_n$ has the same distribution as $\|a\|_2 \gamma_1$, where $\|a\|_2 = \Big(\sum_{n=1}^N |a_n|^2\Big)^{1/2}$. In particular,
\begin{equation}\label{eq:complexGausprop}
\Big(\E\Big|\sum_{k\geq 1} \gamma_n a_n\Big|^p\Big)^{1/p} = \|\gamma_1\|_{L^p(\Omega)} \Big(\sum_{n\geq 1} |a_n|^2\Big)^{1/2}.
\end{equation}
Let $\K$ be either $\R$ or $\C$. A random variable $\gamma:\O\to \K$ is called a {\em (complex) Gaussian random variable} if $\gamma\in L^2(\O)$ and $\E(\gamma)=0$ and $\gamma/(\E|\gamma|^2)^{1/2}$ is a (complex) standard Gaussian random variable. Note that all our Gaussian random variables are centered by definition.

Let $X$ be a (complex) Banach space. A strongly measurable mapping $\xi:\O\to X$ is called a {\em (complex) Gaussian random variable} if for all $x^*\in X^*$, $\lb \xi, x^*\rb$ is a (complex) Gaussian random variable. Observe that if $X$ is a complex Banach space, then we can consider $X$ as a real Banach space and denote this by $X_{\R}$. Let $X'$ be the dual space of $X_{\R}$. One can show that for every function $x'\in X'$ there exists a unique $x^*\in X^*$ such that $x' = {\rm Re}(x^*)$. Moreover, one can define $x^*:X\to \C$ as $\lb x,x^*\rb = \lb x, x'\rb - i \lb i x, x'\rb$. Now if $\xi:\O\to X$ is a complex Gaussian random variable, it follows that $\xi$ can be viewed as a real Gaussian random variable with values in $X_{\R}$. Indeed, for any $x'\in X'$, let $x^*\in X^*$ be as before. Then one has $\lb \xi, x'\rb = {\rm Re}(\lb \xi, x^*\rb)$, and the latter is a real Gaussian random variable.

From the above discussion one can deduce that all results on real Gaussian random variables, have a complex version. Of particular interest is the weak variance of a complex Gaussian random variable.
Define the {\em complex and real weak variances} by
\[\sigma_\R(\xi) = \sup\{(\E|\lb \xi,x'\rb|^2)^{1/2}: x'\in X', \|x'\|\leq 1\},\]
\[\sigma_\C(\xi) = \sup\{(\E|\lb \xi,x^*\rb|^2)^{1/2}: x^*\in X^*, \|x^*\|\leq 1\}.\]
\begin{lemma}\label{lem:complexgaus}
Let $\xi:\O\to X$ be a complex Gaussian. Then
\begin{align*}
\sigma_\C^2(\xi) = \frac12 \sigma_{\R}^2(\xi).
\end{align*}
\end{lemma}
\begin{proof}
Given $x^*\in X^*$ we can write $\lb \xi, x^*\rb = 2^{-1/2} a (\gamma_1 + i \gamma_2)$ with $a\geq 0$ and $\gamma_1, \gamma_2$ independent standard Gaussian random variables. Then one has $\E|\lb \xi, x^*\rb|^2 = a^2$. Letting $x' = \Re(x^*)$, one has $\|x^*\| = \|x'\|$ and $\E|\Re(\lb \xi, x'\rb)|^2 = a^2/2$. This proves $\sigma_\C^2(\xi) \leq \frac12 \sigma_{\R}^2(\xi)$. To prove the converse let $x'\in X^*$, and define $x^*:X\to \C$ as in the discussion before the lemma. Then $x^*\in X^*$ with $\Re(x^*) = x'$, hence $\|x^*\| = \|x'\|$. The first part implies $\sigma_\C^2(\xi) \geq \frac12 \sigma_{\R}^2(\xi)$, and the result follows.
\end{proof}

The following result follows from immediately from the real setting in \cite{HV} and Lemma \ref{lem:complexgaus}. It is a crucial ingredient in the proof of our main result Theorem \ref{thm:main}.
\begin{proposition}\label{prop:HV}
Let $X$ be a complex Banach space. Let $(\xi_n)_{n\geq 1}$ be an $X$-valued
centered multivariate complex Gaussian random variables with first
moments $(m_n)_{n\geq 1}$ and {\em complex} weak variances $(\sigma_n)_{n\geq 1}$.
Let $m=\sup_{n\geq 1}m_n$. Then
\[\begin{aligned}
\E\sup_{n\geq 1} \|\xi_n\| \leq m + 3\sqrt{2} \rho_{\Theta}((\sigma_{n})_{n\geq
1}).
\end{aligned}\]
\end{proposition}
Here $\rho_{\Theta}(a_n)_{n\geq1})$ denotes the Luxemburg norm in the Orlicz space $\ell^{\Theta}$, where $\Theta:\R_+\to \R_+$ is given by $\Theta(0)=0$ and
\[\Theta(x) = x^2 \exp\Big(-\frac{1}{2x^2}\Big) \ \ \text{ for } x>0.\]
We refer to \cite{HV} and references therein for details. For the purposes below it is sufficient to recall that (see \cite[Example 2.1]{HV}) if $a_n = \alpha^n$ with $\alpha\in [1/2,1)$, then
\begin{equation}\label{eq:rhoan}
 \rho_{\Theta}((a_n)_{n\geq 1}) \eqsim \sqrt{\log[(1-\alpha)^{-1}]}.
\end{equation}
Moreover, if $a_n = \alpha^n$ with $\alpha\in (0,1/2)$, then $a_n \leq 2^{-n}$ and therefore \eqref{eq:rhoan} yields
\begin{equation}\label{eq:rhoan2}
 \rho_{\Theta}((a_n)_{n\geq 1}) \leq \rho_{\Theta}((2^{-n})_{n\geq 1}) \eqsim \sqrt{\log(2)}.
\end{equation}

\begin{remark}
In the case that the elements $(\xi_n)_{n\geq 1}$ are independent, $\E\sup_{n\geq 1} \|\xi_n\|$ is equivalent to $m+\rho_{\Theta}((\sigma_{n})_{n\geq 1})$. Moreover, the independence assumption can even be weakened (see \cite{VerCor}).
\end{remark}

Finally, recall that $M>0$ is {\em a median} of a centered (complex) Gaussian random variable $\xi:\O\to X$ if 
\[\P(\|\xi\|\leq M)\geq 1/2 \ \ \text{and} \ \ \P(\|\xi\|\geq M)\geq 1/2.\]

\section{White noise and Besov spaces on $\T^d$}
Let $(\O,\A,\P)$ be a probability space.

\begin{definition}[White noise]
A random variable $W:\O\to \D'(\T^d)$ is a called {\em Gaussian white noise} if
\begin{enumerate}
\item for each $f\in \D(\T^d)$, the random variable $\omega\mapsto \lb f, W(\omega)\rb$ is a complex Gaussian random variable,
\item for all $f,g\in \D(\T^d)$ one has
\[\E(\lb f,W\rb \overline{\lb g,W\rb}) = (f,g)_{L^2(\T^d)}.\]
\end{enumerate}
\end{definition}

If $W:\O\to \D'(\T^d)$ is a Gaussian white noise, then for all $f\in \D(\T^d)$ one has
\[\|\lb f,W\rb\|_{L^2(\O)} = \|f\|_{L^2(\T^d)}.\]
Therefore, the mapping $f\mapsto \lb f,W\rb$ uniquely extends to a bounded linear operator $\mathcal{W}: L^2(\T^d)\to L^2(\O)$. Moreover, $\mathcal{W}$ satisfies
\begin{enumerate}
\item for each $f\in L^2(\T^d)$, the random variable $\mathcal{W} f$ is a complex Gaussian random variable,
\item for all $f,g\in \D(\T^d)$ one has
\[\E(\mathcal{W}(f) \overline{\mathcal{W}(g)} ) = (f,g)_{L^2(\T^d)}.\]
\end{enumerate}
Indeed, since $\D(\T^d)$ is dense in $L^2(\T^d)$, (1) follows from the standard fact that the $L^2(\O)$-limit of a sequence of complex Gaussian random variables is again a complex Gaussian random variable, and (2) follows by an approximation argument.


The following lemma is obvious from the properties (1) and (2) and the fact that a complex Gaussian vector is determined by its covariance structure.
\begin{lemma}\label{lem:equidistr}
Assume $V,W:\O\to \D'(\T^d)$ are both Gaussian white noises with corresponding operators $\mathcal{V}, \mathcal{W}: L^2(\T^d)\to L^2(\O)$. Then $\big(\mathcal{V} f: f\in L^2(\T^d)\big)$ and $\big(\mathcal{W} f: f\in L^2(\T^d)\big)$ are identically distributed.
\end{lemma}

For $k\in \Z^d$, let $e_k:\T^d\to \C$ be given by $e_k(x) = e^{ik\cdot x}$.
\begin{proposition}\label{prop:white}
Let $(\gamma_k)_{k\in \Z^d}$ be a sequence of independent standard complex Gaussian random variables. Let $W:\O\to \D'(\T^d)$ be defined by
$W = \sum_{k\in \Z^d} \gamma_k e_k$ in $\D'(\T^d)$, i.e. $\lb f, W\rb = \sum_{k\in \Z} \gamma_k \hat{f}(-k)$. Then $W$ is a Gaussian white noise.
\end{proposition}
\begin{proof}
By \cite[equation (3.7)]{LT} there is a set $\O_0\in \A$ with $\P(\O_0)=1$ such that
\[\xi : = \sup_{k\in \Z^d} |\gamma_k|(\log(|k|^d+1))^{-1/2} = \sup_{n\geq 0} \sup_{|k|=n} |\gamma_k|(\log(n^d+1))^{-1/2}<\infty\]
on $\O_0$.  Therefore, for any $f\in \D(\T^d)$ and $\omega\in\O_0$ one has
\begin{align*}
|\lb f, W(\omega)\rb| & = \Big|\sum_{k\in \Z} \hat{f}(-k) \gamma_k(\omega)\Big| \leq \sum_{k\in \Z} |\hat{f}(-k)| |\gamma_k(\omega)|
\\ & \leq \xi(\omega) \sum_{n\geq 0}\sum_{k\in \Z} (\log(|k|^d+1))^{1/2} |\hat{f}(-k)|
\\ & \leq \xi(\omega) c_2 \sum_{k\in \Z} (\log(|k|^d+1))^{1/2} (|k|+1)^{-2} \leq C \xi(\omega).
\end{align*}
where $c_2 = \sup_{k\in \Z^d}(1+|k|)^2 |\hat{f}(-k)|$ and $C = c_2 \sum_{k\in \Z} (\log(|k|^d+1))^{1/2} (|k|+1)^{-2}$. It is well-known that $c_2\leq K \sum_{|\alpha|\leq 2} \|f \|_{\alpha}$ for some constant $K$ (see \cite[Theorem 3.2.9]{GrafClas}). Therefore, it follows that $W(\omega)\in \D'(\T^d)$ for all $\omega\in \O_0$. To see that it is a Gaussian white noise, note that $\lb f, W\rb$ is a complex Gaussian and \[\E(\mathcal{W}(f) \overline{\mathcal{W}(g)} )  = \sum_{k\in \Z} \sum_{j\in \Z} \E  (\gamma_k \overline{\gamma_j}) \hat{f}(-k) \overline{\hat{g}(-k)} =  \sum_{k\in \Z} \hat{f}(-k) \overline{\hat{g}(-k)} = (f,g)_{L^2(\T^d)}.\]
\end{proof}

\begin{theorem}\label{thm:main}
Let $W:\O\to \D'(\T^d)$ be a Gaussian white noise.
\begin{enumerate}
\item For all $p\in [1, \infty)$ one has $\displaystyle \P\big(W\in B^{-d/2}_{p,\infty}(\T^d)\big) = 1$.
\item For all $p\in [2, \infty)$ there exists a constant $c_{p,d}$ such that
\[\P\big(\|W\|_{B^{-d/2}_{p,\infty}(\T^d)}\geq c_{p,d}\big) =1.\]
\item For any $p,q\in [1, \infty]$ and $s<-d/2$ one has
$\displaystyle \P\big(W\in B^{s}_{p,q}(\T^d)\big)=1$.
\item For any $p\in [1, \infty]$ and $q\in [1, \infty)$ one has
$\displaystyle \P\big(W\in B^{-d/2}_{p,q}(\T^d)\big)=0$.
\item For any $p,q\in [1, \infty]$ and $s>-d/2$ one has
$\displaystyle P\big(W\in B^{s}_{p,q}(\T^d)\big)=0$.
\item One has $\displaystyle \P\big(W\in B^{-d/2}_{\infty,\infty}(\T^d)\big)=0$.
\end{enumerate}
\end{theorem}

\begin{remark}\label{rem:4cases} Some remarks on the theorem:
\begin{enumerate}
\item[(i)] We do not know whether (2) holds for $p\in [1, 2)$ as well. Applying the Hausdorff Young's inequality instead of Parseval's identity is not sufficient here.


\item[(ii)] The proof shows that one can take $c_{2,d} = \sqrt{2^{3d/2} - 2^{-3d/2}}$ in (2).
\end{enumerate}
\end{remark}

\begin{proof}[Proof of Theorem \ref{thm:main}]
Before we present the proof we reduce to a particular choice of the Gaussian white noise.
Let $W = \sum_{k\in \Z^d} \gamma_k e_k$ in $\D'(\T^d)$. Then by Proposition \ref{prop:white}, $W$ is a Gaussian white noise. Let $V:\O\to \D'(\T^d)$ be an arbitrary white noise, and let $\mathcal{W}$ and $\mathcal{V}$ be as in Lemma \ref{lem:equidistr}. Fix $p,q\in [1, \infty]$. As in \eqref{eq:fjalsconv} let
\begin{equation}\label{eq:WjVj}
W_j(x) = \lb \varphi_{j,x}, W \rb = \mathcal{W}(\varphi_{j,x}) \text{  and  } V_j(x) = \lb \varphi_{j,x}, V \rb = \mathcal{V}(\varphi_{j,x}) \ \ j\in \N.
\end{equation}
Now one has
\begin{equation}\label{eq:BesovWV}\begin{aligned}
\|W\|_{B^{s}_{p,q}(\T^d)} &= \|(2^{js} W_j)_{j\geq 0}\|_{\ell^q(L^p(\T^d))}  \text{  and  }
\\ \|V\|_{B^{s}_{p,q}(\T^d)} & = \|(2^{js} V_j)_{j\geq 0}\|_{\ell^q(L^p(\T^d))}
\end{aligned}
\end{equation}
It follows from \eqref{eq:WjVj} and Lemma \ref{lem:equidistr} that $(W_j)_{j\geq 0}$ and $(V_j)_{j\geq 0}$ $(V_j(x): x\in \T^d, j\in \N)$ and $(W_j(x): x\in \T^d, j\in \N)$ are $\P$-identically distributed . Therefore, since each $V_j$ and $W_j$ is continuous, an approximation with Riemann sums shows that for all $n\in \N$, $(2^{js} \|W_j\|_{L^p(\T^d)})_{j=1}^n$  and $(2^{js} \|V_j\|_{L^p(\T^d)})_{j=1}^n$ are identically $\P$-distributed. Therefore,
\[\|(2^{js} \|W_j\|_{L^p(\T^d)})_{j=1}^n\|_{\ell^q_n} = \|(2^{js} \|V_j\|_{L^p(\T^d)})_{j=1}^n\|_{\ell^q_n} \text{  in $\P$-distribution}.\]
By \eqref{eq:BesovWV} and monotone convergence one obtains that
\[\|W\|_{B^{s}_{p,q}(\T^d)} = \|V\|_{B^{s}_{p,q}(\T^d)} \text{  in $\P$-distribution}.\]

The above shows that below it suffices to consider the situation where
\[W=\sum_{k\in \Z^d} \gamma_k e_k \text{ in $\D'(\T^d)$}.\]

(1): \
For each $j\geq 0$, one has
\[W_j(x) = \sum_{k\in \Z^d} \varphi_j(k) \gamma_k e^{ik\cdot x},  \ \ j\in \N.\]
To show that $W\in B^{-d/2}_{p,\infty}(\T^d)$ note that
\[\E \|W\|_{B^{-d/2}_{p,\infty}(\T^d)} = \E(\sup_{j\geq 0} 2^{-jd/2} \|W_j\|_{L^p(\T^d)})<\infty.\]
To estimate the latter note that by Proposition \ref{prop:HV} one has
\begin{equation}\label{eq:maxest}
\E(\sup_{j\geq 0} 2^{-jd/2} \|W_j\|_{L^p(\T^d)}) \leq  \sup_{j\geq 0} 2^{-jd/2} \E \|W_j\|_{L^p(\T^d)} + 3 \sqrt{2} \rho_{\Theta}((\sigma_j)_{j\geq 0}).
\end{equation}

By \eqref{eq:complexGausprop} one has
\begin{align*}
\E \|W_j\|_{L^p(\T^d)} & \leq  (\E \|W_j\|_{L^p(\T^d)}^p)^{1/p}  = \Big(\int_{\T^d} \|W_j(x)\|_{L^p(\O)}^p \, dx\Big)^{1/p}\\ &  = (2\pi)^{d/p} \|\gamma_1\|_{L^p(\O)} \Big(\sum_{k\in \Z^d} |\varphi_j(k) e_k|^2\Big)^{1/2}
\\ & = \|\gamma_1\|_{L^p(\O)} (2\pi)^{d/p} \Big(\sum_{k\in \Z^d} |\varphi_j(k)|^2\Big)^{1/2}.
\end{align*}
Since $\varphi_j(k) = 0$ if $|k|>2^{j+1}$ it follows that  $\E \|W_j\|_{L^p(\T^d)}\leq \|\gamma_1\|_{L^p(\O)} (2\pi)^{d/p} 2^{jd/2}$ and therefore
\[\sup_{j\geq 0} 2^{-jd/2} \E \|W_j\|_{L^p(\T^d)} \leq \|\gamma_1\|_{L^p(\O)} (2\pi)^{d/p}.\]

To estimate the complex weak variances $\sigma_j$ of $2^{-j d/2} W_j$, let $p'\in (1, \infty]$ be such that $\frac{1}{p} + \frac{1}{p'}=1$. Note that for any $f\in \D(\T^d)$ with $\|f\|_{L^{p'}(\T^d)}\leq 1$ one has
\[\E|\lb W_j, f\rb|^2 = \E \Big|\sum_{k\in \Z^d} \varphi_j(k) \gamma_k \hat{f}(-k)\Big|^2 = \sum_{k\in \Z^d} |\varphi_j(k)|^2 |\hat{f}(-k)|^2 \leq \sum_{|k|\leq 2^{j+1}} |\hat{f}(-k)|^2.\]
First assume $p\in [2, \infty)$. In that case by H\"older's inequality with $\frac12 = \frac1r + \frac1p$ and the Hausdorff-Young inequality we obtain
\begin{align*}
\Big(\sum_{|k|\leq 2^{j+1}} |\hat{f}(-k)|^2\Big)^{1/2} & \leq 2^{(j+3)d/r} \Big(\sum_{k\in \Z^d} |\hat{f}(-k)|^{p}\Big)^{1/p}\\ & \leq 2^{(j+3)d/r} \|f\|_{L^{p'}(\T^d)}\leq 2^{(j+3)d/r} = 2^{-(j+3)d/p} 2^{(j+3)d/2}.
\end{align*}
It follows that $\sigma_j \leq 2^{3d/2} 2^{-(j+3)d/p}$ if $p\in [2, \infty)$ and therefore by \eqref{eq:rhoan} and \eqref{eq:rhoan2},
\[\rho_{\Theta}((\sigma_j)_{j\geq 1}) \leq 2^{d/2-d/p} \rho_{\Theta}((2^{-jd/p})_{j\geq 0})<\infty.\]
Next assume $p\in [1, 2)$. Then by H\"older's inequality one has
\[\Big(\sum_{|k|\leq 2^{j+1}} |\hat{f}(-k)|^2\Big)^{1/2}\leq \|f\|_{L^2(\T^d)} \leq C_d \|f\|_{L^{p'}(\T^d)} = C_d,\]
where $C_d = (2\pi)^{\frac1p - \frac12}$
Therefore, $\sigma_j \leq C_d 2^{-jd/2}$ if $1\leq p<2$, and again $\rho_{\Theta}((\sigma_j)_{j\geq 1})<\infty$.
Now (1) follows from \eqref{eq:maxest}.


(2): \ Since $\|W\|_{B^{-d/2}_{p,\infty}(\T^d)}\geq (2\pi)^{d/p-d/2} \|W\|_{B^{-d/2}_{2,\infty}(\T^d)}$ it suffices to estimate $\|W\|_{B^{-d/2}_{2,\infty}(\T^d)}$ from below by a constant. For each $j\geq 1$ let
\[S_j = \{k\in \Z^d: 2^{j-\frac12}<|k|\leq 2^{j+\frac12}\}.\]
Then $\varphi_j= 1$ on $S_j$, and the $(S_j)_{j\geq 1}$ are pairwise disjoint.
By Parseval's identity one has
\begin{align*}
2^{-jd}\|W_j\|_{L^2(\T^d)}^2 & = 2^{-jd}  \sum_{k\in \Z} |\varphi_j(k) \gamma_k|^{2}
\geq 2^{-jd}  \sum_{k\in S_j} |\gamma_k|^{2}.
\end{align*}
We claim that $\lim_{j\to \infty} 2^{-jd}  \sum_{k\in S_j} |\gamma_k|^{2} = 2^{d/2} - 2^{-d/2}$ almost surely. Indeed, one can write
\begin{align*}
2^{-jd}  \sum_{k\in S_j} |\gamma_k|^{2} & = 2^{-jd}  \sum_{|k|\leq 2^{j+\frac12}} |\gamma_k|^{2} - 2^{-jd}   \sum_{|k|\leq 2^{j-\frac12}} |\gamma_k|^{2}
\\ & = 2^{3d/2} 2^{-(j+\frac32)d} \sum_{|k|\leq 2^{j+\frac12}} |\gamma_k|^{2} - 2^{-3d/2} 2^{-(j+\frac12)d}   \sum_{|k|\leq 2^{j-\frac12}} |\gamma_k|^{2},
\end{align*}
and by the strong law of large numbers the latter converges almost surely to $2^{3d/2} - 2^{-3d/2}$
Therefore, we can conclude that
\[\|W\|_{B^{1/2}_{2,\infty}(\T^d)} = \sup_{j\geq 0} 2^{-jd/2}\|W_j\|_{L^2(\T^d)}\geq \sqrt{2^{3d/2} - 2^{-3d/2}}.\]
almost surely.

%

(3): This is clear from the fact that $B^{-d/2}_{p,\infty}(\T^d)\subseteq B^{s}_{p,q}(\T^d)$ for any $s<-d/2$.

(4): \ From the proof of (2) one immediately sees that for all $p\in [2, \infty)$ and $q\in [1, \infty)$, $\|W\|_{B^{1/2}_{p,q}(\T^d)} = \infty$ almost surely. Indeed, let $N\in \A$ be the set of all $\omega\in \O$ such that $\|W(\omega)\|_{B^{1/2}_{p,q}(\T^d)} <\infty$. Then for all $\omega\in N$ one has $\displaystyle \lim_{j\to \infty} 2^{-jd/2}\|W_j(\omega)\|_{L^p(\T^d)}=0$. In (2) we have seen that $\displaystyle \lim_{j\to \infty} 2^{-jd/2}\|W_j\|_{L^p(\T^d)}\geq c_{p,d}$ on a set $\O_0$ of probability $1$. Therefore, $N\subset \O\setminus \O_0$, and hence $N$ is a zero set.

If $p=\infty$ and $q\in [1, \infty)$, then taking any $2\leq r<\infty$ one obtain that
$\|W\|_{B^{1/2}_{\infty,q}(\T^d)}\geq C_d \|W\|_{B^{1/2}_{r,q}(\T^d)} =\infty$ almost surely.

If $1\leq p\leq 2$ and $q\in [1, \infty)$ we cannot use (2) and we need to argue in a different way. Fix integers $n\geq m\geq 0$. For each $\omega\in \O$ one has $\|W(\omega)\|_{B^{1/2}_{p,q}(\T^d)}<\infty$ if and only if  $(2^{-jd /2} W_j(\omega))_{j=1}^\infty$ has a finite $\ell^q(L^p(\T^d))$-norm. The latter is equivalent to the convergence of $\sum_{j\geq 0} \xi_j(\omega)$, where  $\xi_j(\omega) = 2^{-jdq /2}\| W_j(\omega)\|_{L^p(\T^d)}^q$. Therefore, by a standard $0/1$ law argument (see \cite[Corollary 3.14]{Kal}) it follows that $\P(\|W\|_{B^{1/2}_{p,q}(\T^d)}<\infty)$ is either $0$ or $1$. Assume $\P(\|W\|_{B^{1/2}_{p,q}(\T^d)}<\infty)=1$. Then one can show that $(2^{-jd /2} W_j)_{j=1}^\infty$ is a vector-valued Gaussian random variable. In particular, it has finite $r$-th moment for all $r<\infty$ (see \cite[Corollary 3.2]{LT}).
By the Kahane-Khintchine inequality (see \cite[Corollary 3.2]{LT}) and (1) we obtain
\[(\E \|W_j\|_{L^p(\T^d)}^q)^{1/q} \eqsim_{p,q} (\E \|W_j\|_{L^p(\T^d)}^p)^{1/p} = \|\gamma_1\|_{L^p(\O)} (2\pi)^{d/p} \Big(\sum_{k\in \Z^d} |\varphi_j(k)|^2\Big)^{1/2}.\]
Since
\[\Big(\sum_{k\in \Z^d} |\varphi_j(k)|^2\Big)^{1/2} \geq \Big(\sum_{k\in S_{j}} 1 \Big)^{1/2}\eqsim 2^{jd}.\]
it follows that
\begin{align*}
\E \|(2^{-jd q/2} W_j)_{j=1}^\infty\|_{\ell^q(L^p(\T^d))}^q & = \sum_{j\geq 0} 2^{-jd q/2}\E \|W_j\|_{L^p(\T^d)}^q
\\ & \gtrsim_{p,q} \|\gamma_1\|_{L^p(\O)}^q (2\pi)^{dq/p} \sum_{j\geq 0} 2^{-jd q/2} 2^{ j d q/2}.
\end{align*}
Clearly, the latter series is infinite and this gives the desired contradiction.

(5): Fix $s>-d/2$. Since $B^{s}_{p,q}(\T^d)\subseteq B^{s}_{1,\infty}(\T^d)\subseteq B^{-d/2}_{1,1}(\T^d)$ for all $p,q\in [1, \infty]$, it suffices to show that $W\notin B^{-d/2}_{1,1}(\T^d)$ almost surely. However, this has already been proved in (4) and therefore (5) follows.

(6): Since each $W_j$ is continuous, one has
\[\|W_j\|_{L^\infty(\T)} = \|W_j\|_{C(\T^d)}\geq |W_j(0)|= 2^{-jd/2} \Big|\sum_{k\in \Z} \varphi_j(k)  \gamma_k\Big|.\]
Since the supports of $(\varphi_{3j})_{j\geq 1}$ are disjoint one has that $(W_{3j}(0))_{j\geq 1}$ are independent complex Gaussian random variables with values in $\R$. Moreover,
\begin{align*}
2^{-3jd}\E|W_{3j}(0)|^2 & = 2^{-3jd} \sum_{k\in \Z} |\varphi_{3j}(k)|^2 \geq 2^{-3jd} \sum_{2^{3j-\frac12} \leq |k|\leq 2^{3j+\frac12}} 1^2
\geq c,
\end{align*}
where $c$ is independent of $j$. Therefore,
\[\|W\|_{B^{-d/2}_{\infty,\infty}(\T^d)}^2 \geq \sup_{j\geq 1} 2^{-3jd} \|W_{3j}\|_{L^\infty(\T)}^2 \geq \sup_{j\geq 1} 2^{-3jd}  |W_{3j}(0)|^2 = \infty,\]
almost surely.
\end{proof}


The following result characterizes for which exponents $(s,p,q)$ one has $W\in B^{s}_{p,q}(\T^d)$ almost surely.
\begin{corollary}
Let $W:\O\to \D'(\T^d)$ be a Gaussian white noise. For exponents $(s, p, q)\in \R\times[1, \infty]\times[1, \infty]$ the following are equivalent: \begin{enumerate}
\item $W\in B^{s}_{p,q}(\T^d)$ almost surely
\item $\Big( s<-d/2 \text{ and } p,q\in [1, \infty] \Big)$ or $\Big( s=-d/2 \text{ and } p\in [1, \infty) \text{ and } q=\infty \Big)$
\end{enumerate}
\end{corollary}
\begin{proof}
This follows from Theorem \ref{thm:main}.
\end{proof}

Another consequence is on the tail behavior of $\|W\|_{B^{-d/2}_{p,\infty}(\T^d)}$. 
\begin{corollary}\label{cor:tail}
Let $W:\O\to \D'(\R^n)$ be a white noise and let $p\in [1, \infty)$.
there are constants $M,\sigma>0$ depending on $p$ such that for every $r>0$ the following inequality holds:
\begin{equation}\label{eq:Wtail}
\P\big( \big| \|W\|_{B^{-d/2}_{p,\infty}(\T^d)} - M\big|>r \big)\leq \exp(-r^2/(4\sigma^2)).
\end{equation}
Here for $M$ one can take the median of $\|W\|_{B^{-d/2}_{p,\infty}(\T^d)}$ and for $\sigma$ one can take 
\[\sigma = \left\{
    \begin{array}{ll}
      2^{3d/2} 2^{-3d/p} & \hbox{if $p\in [2,\infty)$;} \\
      (2\pi)^{\frac{d}{p} - \frac{d}{2}}, & \hbox{if $p\in [1,2]$.}
    \end{array}
  \right.
\]
\end{corollary}

In particular, this result implies that $W$ satisfies 
\begin{equation}\label{eq:expint}
\E \exp\Big(\frac{1}{4\alpha^2}\|W\|_{B^{-d/2}_{p,\infty}(\T^d)}^2\Big)<\infty
\end{equation}
for all $\alpha>\sigma$ (see \cite[Corollary 3.2]{LT}).

\begin{proof}
By Theorem \ref{thm:main} we can define a mapping $Z:\O\to B^{-d/2}_{p,\infty}(\T^d)$ by $Z(\omega) = W(\omega)$.
As in \cite[Theorem 5.1]{HV} one can show that $Z$ defined by $Z(\omega) = W(\omega)$ is a Gaussian random variable in the sense of \cite{LT}. However, note that $Z$ does not take values in a separable subset of $B^{-d/2}_{p,\infty}(\T^d)$ (see \cite[p. 60-61]{LT}). Now \eqref{eq:Wtail} follows from \cite[Lemma 3.1]{LT}. Moreover, the choice for $\sigma$ follows from the fact that one can take $\sigma = \max_j \sigma_j$, where $\sigma_j$ is as in the proof of Theorem \ref{thm:main}. Note that in \cite[Lemma 3.1]{LT} one has to take into account Lemma \ref{lem:complexgaus}, because we consider the complex situation and this gives a number $4$ instead of $2$ in the exponential function.
\end{proof}

As a consequence one has the following result for the periodic Sobolev spaces. For the definition of the Sobolev space $H^{s,p}(\T^d)$ we refer to \cite[Chapter 3]{TrSch}.

\begin{corollary}
Let $W$ be a Gaussian white noise on $\T^d$. Then
\begin{enumerate}
\item For all $p\in [1, \infty)$ and $s<-d/2$ one has $\displaystyle \P\big(W\in H^{s,p}(\T^d)\big) = 1$.
\item For all $p\in [1, \infty]$ one has
$\displaystyle \P\big(W\in H^{-d/2,p}(\T^d)\big)=0$.
\end{enumerate}
\end{corollary}
\begin{proof}
(1): This follows from Theorem \ref{thm:main} (3) and \eqref{eq:elememb}.

(2): This follows from Theorem \ref{thm:main} (4) and \eqref{eq:elememb}.
\end{proof}

It would be interesting to know whether the results of this section are valid for Gaussian white noises on domains $D\subset \R^d$ and on Riemannian manifolds.

\section{White noise and Fourier--Besov spaces on $\T^d$\label{sec:FourierBesov}}

Let $p,q\in [1, \infty]$ and $s\in \R$. For a distribution $f\in \D(\T^d)$ consider the following Fourier-Besov norm
\[\|f\|_{\hat{b}^s_{p,q}(\T^d)} = \Big(\sum_{j\geq 0} \Big(\sum_{k\in \Z^d} (|k|+1)^{sp}  |\varphi_j(k) \hat{f}(k)|^p\Big)^{q/p}\Big)^{1/q}.\]
Moreover, if $p=\infty$ or $q=\infty$, then one needs to use a supremum norm in this expression.

\begin{definition}
The Fourier-Besov space $\hat{b}^s_{p,q}(\T^d)$ is the space of all $f\in \D(\T^d)$ for which $\|f\|_{\hat{b}^s_{p,q}(\T^d)}<\infty$.
\end{definition}

Let $\phi$ and $(\varphi_j)_{j\geq 0}$ be as in Section \ref{sec:Besov}.
The following is an equivalent norm on $\hat{b}^s_{p,q}(\T^d)$:
\[|\!|\!|f|\!|\!|_{\hat{b}^s_{p,q}(\T^d)} = \Big(\sum_{j\geq 0} \Big(\sum_{k\in \Z^d} (|k|+1)^{sp}  |\varphi_j(k) \hat{f}(k)|^p\Big)^{q/p}\Big)^{1/q},\]
where if $p=\infty$ or $q=\infty$ one has to use the supremum norm again. To show the equivalence of the norms observe that the estimate $|\!|\!|f|\!|\!|_{\hat{b}^s_{p,q}(\T^d)} \leq \|f\|_{\hat{b}^s_{p,q}(\T^d)}$ follows from the fact that $|\varphi_j(k)|\leq 1$ and $\varphi_j(k)=0$ for all $k\in \Z$ which satisfy  $|k|>2^{j+1}$ or $|k|<2^{j-1}$. For the converse direction, we let $\varphi_{-1} =0$. Recall that for all $j\geq 0$ and $k\in \Z^d$ such that $2^{j-1}\leq |k|\leq 2^{j+1}$ one has $\sum_{m=-1}^{1} \varphi_{j+m}(k) = 1$. Therefore, by the triangle inequality in $\ell^q(\ell^p)$ and elementary calculations one sees that
\begin{align*}
\|f\|_{\hat{b}^s_{p,q}(\T^d)} & = \Big(\sum_{j\geq 0} \Big(\sum_{2^{j-1}\leq |k|\leq 2^{j+1}} (|k|+1)^{sp}  \Big|\sum_{m=-1}^{1} \varphi_{j+m}(k)\Big|^p |\hat{f}(k)|^p\Big)^{q/p}\Big)^{1/q}
\\ & \leq  \sum_{m=-1}^1 \Big(\sum_{j\geq 0} \Big(\sum_{2^{j-1}\leq |k|\leq 2^{j+1}} (|k|+1)^{sp}  |\varphi_{j+m}(k)|^p |\hat{f}(k)|^p\Big)^{q/p}\Big)^{1/q}
\\ & \leq  \sum_{m=-1}^1 \Big(\sum_{j\geq 0} \Big(\sum_{k\in \Z} (|k|+1)^{sp}  |\varphi_{j+m}(k)|^p |\hat{f}(k)|^p\Big)^{q/p}\Big)^{1/q}
\\ & \leq  3 \Big(\sum_{j\geq 0} \Big(\sum_{k\in \Z} (|k|+1)^{sp}  |\varphi_{j}(k)|^p |\hat{f}(k)|^p\Big)^{q/p}\Big)^{1/q} = 3 |\!|\!|f|\!|\!|_{\hat{b}^s_{p,q}(\T^d)}.
\end{align*}

From the above discussion and the Hausdorff--Young inequality one obtains the following result which has also been observed in \cite{Oh} for $q=\infty$.
\begin{proposition}
Let $s\in \R$ and $q\in [1, \infty]$. Let $p\in [2, \infty]$ and let $p'\in [1,2]$ satisfy $\frac{1}{p}+\frac{1}{p'}=1$. Then
$B^{s}_{p',q}(\T^d) \hookrightarrow \hat{b}^{s}_{p,q}(\T^d)$. Moreover, if $p=2$, then $B^{s}_{2,q}(\T^d) = \hat{b}^{s}_{2,q}(\T^d)$ with equivalent norms.
\end{proposition}

Let
\[[f]_{\hat{b}^{s}_{p,q}(\T^d)} \Big(\sum_{j\geq 0} 2^{sqj} \Big( \sum_{2^{j-1}\leq |k|\leq 2^{j+1}}|\hat{f}(k)|^p\Big)^{q/p}\Big)^{1/q} ,\]
where one has to use the supremum norm if $q=\infty$. Then
this also defines an equivalent norm on $\hat{b}^s_{p,q}(\T^d)$. Indeed, this follows from the fact that for $2^{j-1}\leq |k|\leq 2^{j+1}$ one has
\[\tfrac12 \cdot 2^{j} \leq  (|k|+1)\leq 4 \cdot 2^{j}.\]
In the calculations below we use this equivalent norm.

In \cite{Oh} it has been proved that the Gaussian white noise as defined in Lemma \ref{prop:white} satisfies $W\in \hat{b}^{s}_{p,q}(\T)$ almost surely as soon as $sp<-1$. The following result is an extension of this result to the sharp exponent and to arbitrary dimensions $d$.

\begin{theorem}\label{thm:main2}
Let $W:\O\to \D'(\T^d)$ be a Gaussian white noise.
\begin{enumerate}
\item For all $p\in [1, \infty)$ one has
$\displaystyle  \P\big(W\in \hat{b}^{-d/p}_{p,\infty}(\T^d)\big) = 1$.
\item For all $p\in [1, \infty)$ there exists a constant $c_{p,d}$ such that
\[\P\big([W]_{\hat{b}^{-d/p}_{p,\infty}(\T^d)}\geq c_{p,d}\big) =1.\]
\item For all $p,q\in [1, \infty]$ and $s<-d/p$ one has
$\displaystyle \P\big(W\in \hat{b}^{s}_{p,q}(\T^d)\big)=1$.
\item For any $p\in [1, \infty]$ and $q\in [1, \infty)$ one has
$\displaystyle \P\big(W\in \hat{b}^{-d/p}_{p,q}(\T^d)\big)=0$.
\item For any $p,q\in [1, \infty]$ and $s>-d/p$ one has
$\displaystyle \P\big(W\in \hat{b}^{s}_{p,q}(\T^d)\big)=0$.
\item One has $\displaystyle \P\big(W\in \hat{b}^{0}_{\infty,\infty}(\T^d)\big)=0$.
\end{enumerate}
\end{theorem}

The analogue Remark \ref{rem:4cases} (iii) holds for the above situation, this time with $X = \hat{b}^{-d/p}_{p,\infty}(\T^d)$ and $p\in [1, \infty)$.

\begin{proof}
As in the proof of Theorem \ref{thm:main}, it is sufficient to consider a Gaussian white noise $W$ as defined in Proposition \ref{prop:white}. This follows again from Lemma \ref{lem:equidistr} and the identity $\hat{W}(k) = \mathcal{W}(e_k)$.

For $p<\infty$ let $w_{j,p} = \Big(\sum_{2^{j-1}\leq |k|\leq 2^{j+1}}|\g_k|^p\Big)^{1/p}$. Then one has
\[[W]_{\hat{b}^{s}_{p,q}(\T^d)} = \|2^{js} w_{j,p}\|_{\ell^q}.\]

(1): \ By the above we can write
\[[W]_{\hat{b}^{-1/p}_{p,\infty}(\T^d)} = \sup_{j\geq 0} 2^{-j d/p} w_{j,p}.\]
Noting that
\begin{align*}
2^{-jd} |w_{j,p}|^p &= 2^{-j d} \sum_{2^{j-1}\leq |k|\leq 2^{j+1}}|\g_k|^p \\ & = 2^d 2^{-(j+1)d} \sum_{|k|\leq 2^{j+1}}|\g_k|^p - 2^{-d} 2^{-(j-1)d} \sum_{|k|<2^{j-1}}|\g_k|^p,
\end{align*}
from the strong law of large numbers we see that
\[\lim_{j\to \infty} 2^{-jd} |w_{j,p}|^p = (2^{d} - 2^{-d})\E |\g_1|^p,\]
almost surely. Therefore, $\sup_{j\geq 0} 2^{-j d/p} w_{j,p}<\infty$  almost surely and this proves (1).

(2): It follows from the proof of (1) that $[W]_{\hat{b}^{-1/p}_{p,\infty}(\T^d)}\geq (2^{d} - 2^{-d}) \|\g_1\|_{L^p(\O)}$ almost surely, and this proves the result.

(3): This follows from (1) in the same way as in Theorem \ref{thm:main}.

(4): This follows from (2) in the same way as in Theorem \ref{thm:main}. This time $p\in [1, 2)$ does not have to be considered separately, because (2) holds for all $p\in [1, \infty)$.

(5): This follows from (2) in the same way as in Theorem \ref{thm:main}.

(6): In this case one has
\[[W]_{\hat{b}^{s}_{p,q}(\T^d)} = \sup_{j\geq 0} \sup_{2^{j-1} \leq |k|\leq 2^{j+1}}|\g_k| = \sup_{k\in \Z^d} |\g_k|.\]
It is well-known that the latter is infinite almost surely (see \cite[equation (3.7)]{LT}).
\end{proof}

The following result characterizes for which exponents $(s,p,q)$ one has $W\in \hat{b}^{s}_{p,q}(\T^d)$ almost surely.
\begin{corollary}
Let $W:\O\to \D'(\T^d)$ be a Gaussian white noise. For exponents $(s, p, q)\in \R\times[1, \infty]\times[1, \infty]$ the following are equivalent: \begin{enumerate}
\item $W\in \hat{b}^{s}_{p,q}(\T^d)$ almost surely
\item $\Big( s<-d/p \text{ and } p,q\in [1, \infty] \Big)$ or $\Big( s=-d/p \text{ and } p\in [1, \infty) \text{ and } q=\infty \Big)$
\end{enumerate}
\end{corollary}
\begin{proof}
This follows from Theorem \ref{thm:main2}.
\end{proof}

Another consequence is on the tail behavior of $\|W\|_{\hat{b}^{-d/p}_{p,\infty}(\T^d)}$.
\begin{corollary}\label{cor:tail2}
Let $W:\O\to \D'(\R^n)$ be a white noise and let $p\in [1, \infty)$.
there are constants $M,\sigma>0$ depending on $p$ such that for every $r>0$ the following inequality holds:
\begin{equation}\label{eq:Wtail2}
\P\big( \big| \|W\|_{\hat{b}^{-d/p}_{p,\infty}(\T^d)} - M\big|>r \big)\leq \exp(-r^2/(4\sigma^2)).
\end{equation}
Here for $M$ one can take the median of $\|W\|_{\hat{b}^{-d/p}_{p,\infty}(\T^d)}$ and for $\sigma$ one can take
\[\sigma = \left\{
    \begin{array}{ll}
      1 & \hbox{if $p\in [2,\infty)$;} \\
      2^{3d/p} 2^{-3d/2}, & \hbox{if $p\in [1,2]$.}
    \end{array}
  \right.
\]
\end{corollary}

Again this result implies that $W$ satisfies the exponential integrability result in  \eqref{eq:expint} with  $B^{-d/2}_{p,\infty}(\T^d)$ replaced by $\hat{b}^{-d/p}_{p,\infty}(\T^d)$.

\begin{proof}
This can be proved in the same way as in Corollary \ref{cor:tail}. To calculate the $\sigma_j$'s in this case, let $G_j:\O\to \ell^p(\Z^d)$ be given by $G_j = 2^{-jd/p}\sum_{|k|\leq 2^{j+1}} \gamma_k u_k$, where $(u_k)_{k\in \Z^d}$ is the standard unit basis of $\ell^p(\Z^d)$. Then for all $a=(a_k)_{k\in \Z^d}$ in $\ell^{p'}(\Z^d)$ with norm $\leq 1$, one has that
\[\E|\lb G_j, a\rb|^2 = 2^{-jd} \sum_{|k|\leq 2^{j+1}} |a_k|^2.\]
Hence, if $p\in [2,\infty)$, then $\|a\|_{\ell^2(\Z^d)}\leq \|a\|_{\ell^{p'}(\Z^d)}\leq 1$,
and therefore, $\sigma_j \leq 2^{-jd/p}$. It follows that $\sigma = \max_j \sigma_j \leq 1$. If $p\in [1, 2)$, then by H\"older's inequality
\[\Big(\sum_{|k|\leq 2^{j+1}} |a_k|^2\Big)^{1/2}\leq 2^{(j+3)d/p} 2^{-(j+3)d/2} \|a\|_{\ell^{p'}(\Z^d)} \leq 2^{(j+3)d/p} 2^{-(j+3)d/2}.\]
Hence $\sigma_j \leq 2^{3d/p} 2^{-(j+3)d/2}$ and it follows that $\sigma = \max_j \sigma_j \leq 2^{3d/p} 2^{-3d/2}$.
\end{proof}

\providecommand{\bysame}{\leavevmode\hbox to3em{\hrulefill}\thinspace}

\end{document}